\documentclass[a4paper,12pt,reqno]{amsart}
\usepackage[utf8]{inputenc}
\usepackage{mathtools}

\makeatletter
\@namedef{subjclassname@2020}{\textup{2020} Mathematics Subject Classification}
\makeatother

\allowdisplaybreaks[1]

\usepackage{amsmath,amsthm,amssymb}
\usepackage{hyperref}
\usepackage{amsfonts}
\usepackage{amsmath}
\usepackage{mathrsfs}
\usepackage{tabto}
\usepackage{changepage} 
\usepackage{fullpage} 
\usepackage{graphicx}
\usepackage{enumitem} 

\numberwithin{equation}{section}
\newtheorem{theorem}{Theorem}[section]
\newtheorem{lemma}[theorem]{Lemma}

\newtheorem{proposition}[theorem]{Proposition}

\newtheorem{question}{Question}
\newtheorem{observation}{Observation}
\newtheorem{remark}[theorem]{Remark}
\newtheorem{example}[theorem]{Example}

\thanks{The research work of the first author is supported by Institute Research Fellowship of BIT Mesra APO/2024-25/49 and the research work of the second author is supported by ANRF(SERB) research grant TAR/2023/000197}

\begin{document}
	
	\title [Dynamical sets of semigroup]{Equality of the dynamical sets of two commuting transcendental entire functions}
	\author[M. Kumari]{Manisha Kumari}
	\address{Department of Mathematics, Birla Institute of Technology Mesra
		Ranchi--835 215, India}
	
	\email{phdam10052.24@bitmesra.ac.in}
	
	\author[D. Kumar]{Dinesh Kumar}
	\address{Department of Mathematics, Birla Institute of Technology Mesra
		Ranchi--835 215, India}
	
	\email{dineshkumar@bitmesra.ac.in }


	
	\begin{abstract}
		
		In this paper, we investigate the dynamics of commuting transcendental entire functions \( f \) and \( g \), where \( g \) is expressed in the form \( af^p + b \) with \( a, b \in \mathbb{C} \), \( p \in \mathbb{N} \), and \( a \neq 0, 1 \). We establish that the escaping set, the filled Julia set, and the bungee set of both \( f \) and \( g \) are identical. As a direct consequence, we find that the Julia sets of \( f \) and \( g \) coincide. Our theorem extends the result established by Poon and Yang in 1998. Furthermore, building on the 2003 work of Wang and Yang, we consider a non-constant polynomial \( Q \) and permutable entire functions \( f \) and \( g \) that satisfy the relation \( Q(g) = aQ(f) + b \), where \( a \neq 0, 1 \) and \( b \in \mathbb{C} \). In this broader context, we also demonstrate that the escaping sets, the filled Julia sets, and the bungee sets of \( f \) and \( g \) are equal.
	\end{abstract}
	

	
	\maketitle
	\section{Introduction}
Let \( f: \mathbb{C} \to \mathbb{C} \) be a transcendental entire function, and let \( f^n \) denote the \( n \)-th iterate of \( f \) for all \( n \in \mathbb{N} \). The dynamical behavior of \( f \) is characterized by partitioning the complex plane into two fundamental sets: the Fatou set and the Julia set.

The Fatou set of \( f \), denoted as \( F(f) \), is defined as follows
\[
F(f) = \{ z \in \mathbb{C} : \{ f^n \}_{n \in \mathbb{N}} \text{ is normal in some neighborhood of } z \}
\]
The complement of the Fatou set, \( J(f) = \mathbb{C} \setminus F(f) \), is known as the Julia set. Points within the Fatou set exhibit stable dynamical behavior, while the Julia set represents chaotic dynamics and sensitive dependence on initial conditions. Basic properties and structural aspects of these sets are well documented in the literature \cite{bergweiler1993iteration}, \cite{carleson2013complex}.
	
In addition to the classical decomposition of the complex plane into Fatou and Julia sets, points can be further classified based on their behavior under iteration. The escaping set of a function \( f \) is denoted by \( I(f) \) and is defined as follows $I(f) = \{ z \in \mathbb{C} : |f^{n}(z)| \to \infty \text{ as } n \to \infty \}.$ This set, introduced by Eremenko in his 1989 paper, consists of all points whose forward orbit tends to infinity. Eremenko also proved that \( I(f) \) is non-empty and that every connected component of its closure, \( \overline{I(f)} \), is unbounded.

The filled Julia set of \( f \) is denoted by \( K(f) \) and is defined as $K(f) = \{ z \in \mathbb{C} : \{ f^{n}(z) \}_{n \in \mathbb{N}} \text{ is bounded} \}.$
That is, this set contains points with bounded orbits. Furthermore, the bungee set of \( f \), denoted \( BU(f) \), is defined as the collection of points whose orbit exhibits mixed dynamical behavior. Specifically, in this case, the orbit has at least two subsequences: one that remains bounded and another that tends to infinity.
	
	%
Two functions \( f \) and \( g \) are said to be commuting (or permutable) if \( f \circ g = g \circ f \). Fatou proved that such permutable rational functions possess identical Fatou sets \cite{beardon2000iteration}. Under specific conditions, Hua and Wang \cite{hua2002dynamics} addressed a question initially posed by Baker \cite{baker1984wandering}, which asks:

\begin{question}
	For two given permutable transcendental entire functions \( f \) and \( g \), do we have \( F(f) = F(g) \)?
\end{question} 

This leads naturally to the question of whether a similar result holds for the escaping set, the filled Julia set, and the bungee set, which is explored in this article.
	
This study explores the dynamics of specific classes of transcendental entire functions and their affine modifications, motivated by recent developments in this field. We have shown that under certain conditions, two commuting entire functions \( f \) and \( g \) share the same filled Julia set, escaping set, and bungee set. Our findings extend the work of Poon and Yang \cite{poon1998dynamical} by establishing that these three sets are equal for a broader class of commuting transcendental entire functions. Additionally, we demonstrate that the Julia sets of these commuting transcendental entire functions are also identical. Our research builds on a result of Wang and Yang \cite[Theorem 1]{yang}, which investigated a non-constant polynomial \( Q \) and permutable entire functions \( f \) and \( g \) that satisfy the functional equation \( Q(g) = aQ(f) + b \), where \( a \neq 0, 1 \) and \( b \in \mathbb{C} \). In a more general framework, we prove that the escaping set, the filled Julia set, and the bungee set of \( f \) and \( g \) are all the same.

\section{The Equality of Dynamical Sets} \label{sec 3}

As demonstrated in \cite[Lemma 2.1]{poon1998dynamical}, if \( f \) and \( g \) are permutable transcendental entire functions, then \( F(f) = F(g) \), where \( g(z) = af(z) + b \) and \( a, b \in \mathbb{C} \) are constants with \( a \neq 0 \). Additionally, it is required that \( |a| = 1 \). This leads to the exploration of analogous conclusions regarding the filled Julia set, the escaping set, and the bungee set, which can be naturally examined within a more general framework.

The following observation will be referenced frequently in the subsequent sections, as it outlines how affine mappings act on both bounded and unbounded sets.

\begin{observation} \label{lem 7}
Let \( \psi(z) = az + b \) with \( a, b \in \mathbb{C} \) and \( a \neq 0 \). Then \( \psi \) sends every bounded subset of \( \mathbb{C} \) to a bounded set, and every unbounded subset to an unbounded set.
\end{observation}
Before proving the next result, we first establish the following lemma, which clarifies the connection between the functions \( f \), \( g \), and the affine function \( \psi(z) = az + b \) where \( a \neq 0 \) and \( b \in \mathbb{C} \).

\begin{lemma}\label{lem3}
	The functions \( f(z) \) and \( g(z) = af(z) + b \) commute if and only if \( f(z) \) and \( az + b \) commute.
\end{lemma}

\begin{proof}
	We have \( g(z) = af(z) + b \), and let \( \phi(z) = az + b \) with \( a \neq 0 \). We want to prove
		\[
	f \circ g = g \circ f \quad \Longleftrightarrow \quad f \circ \phi = \phi \circ f.
	\]
	
	First, assume that \( f \circ \phi = \phi \circ f \), which means $f(az + b) = a f(z) + b, \quad \forall z \in \mathbb{C}.$ Substituting \( f(z) \) in place of \( z \) gives us $	f(a f(z) + b) = a f(f(z)) + b,$ which is exactly the condition for \( f \circ g = g \circ f \). Thus, we have shown that \( f \circ \phi = \phi \circ f \) implies \( f \circ g = g \circ f \).
	Conversely, assume that \( f \circ g = g \circ f \). We need to show that:
	$
	f \circ \phi = \phi \circ f.
	$ Let \( F(z) = f(az+b) - af(z) - b \). Then \( F \) is clearly an entire function. For every \( z \), we have $F(f(z)) = f(af(z) + b) - a f(f(z)) - b = 0.$
		Therefore, \( F \) vanishes on the set \( f(\mathbb{C}) \). Since \( f \) is a transcendental entire function, by Picard’s theorem, we have
		\[
	f(\mathbb{C}) = \mathbb{C} \quad \text{or} \quad f(\mathbb{C}) = \mathbb{C} \setminus \{\beta\}
	\]
	for some \( \beta \in \mathbb{C} \).
This implies that the zero set of \( F \) includes all of \( \mathbb{C} \) except possibly for a single point. By the identity theorem, we conclude that
	\[
	F \equiv 0.
	\]
	As a result, we obtain the relation
$f(az+b) = af(z) + b,$
	which can be rewritten as
	$f \circ \phi = \phi \circ f.$
	This completes the proof.
		\end{proof}
	\begin{remark}
	More generally, \( f \) and \( g \) are two commuting transcendental entire functions, where \( g = af^p + b \) with \( a, b \in \mathbb{C} \), \( p \in \mathbb{N} \), and \( a \neq 0, 1 \), then we have $f(az + b) = af(z) + b.$ Also, $|a|=1.$
	\end{remark}
	
	We now consider the transcendental entire function $f(z) = \sin z + p(z),$ where $p(z)$ is a polynomial. We need to determine the affine maps of the form:
$\phi(z) = az + b$ that satisfy the condition:
$f \circ \phi = \phi \circ f.$ That is, we need to satisfy the following equation
	\begin{equation}
		\sin(az+b) + p(az+b) = a(\sin z + p(z)) + b.
		\tag{1}
	\end{equation}
	This leads to the conclusion that $	a = \pm 1.$\\
	\noindent
	\textbf{Case 1: \( a = 1 \).} In this scenario, we have
	\begin{equation}
		\sin(z + b) + p(z + b) = \sin z + p(z) + b.
		\tag{2}
	\end{equation}
	As \(\sin(z + b)\) involves both \(\sin z\) and \(\cos z\), for equation (2) to hold true, we must have 
	$b = 2k\pi, \quad k \in \mathbb{Z}.$
	Thus, we can derive that 
$p(z + 2k\pi) = p(z) + 2k\pi.$
	Given that \( p \) is a polynomial, the expression 
	$p(z + 2k\pi) - p(z)$  is constant only when \( p \) is linear, which leads us to 
$p(z) = z + d, \quad d \in \mathbb{C}.$ As a result, when \( a = 1 \), we arrive at 
	$f(z) = \sin z + z + d.$ Additionally, it follows that 
	\[
	f(z + 2k\pi) = f(z) + 2k\pi.
	\]
	\noindent
		\textbf{Case 2: $a=-1$.} In this case,
		\[
		\sin(-z+b)+p(-z+b)
		=
		-\sin z-p(z)+b.
		\]
		As
		$
		\sin(-z+b)=-\sin z
		$
		holds when $b=2k\pi$, we obtain
		\begin{equation}\label{eq 3}
			p(-z+2k\pi)=-p(z)+2k\pi.
			\tag{3}
		\end{equation}
		Let
		$
		z=w+k\pi.
		$ Then equation \eqref{eq 3} reduces to $p(-w+k\pi)+p(w+k\pi)=2k\pi.$
	Consequently, the function $p(z) - k\pi$ is an odd function about \( z = k\pi \). For example, when \( k = 0 \), we have $p(-z) = -p(z),$ indicating that \( p \) is an odd function, and therefore, \( f \) is also an odd function. As a result, \( -f \) commutes with \( f \).
	We will now demonstrate that $I(f)$, $K(f)$ and $BU(f)$ equal
	$I(g)$, $K(g)$ and $BU(g)$ respectively, where
	$f(z) = \sin z + p(z)$ and $g(z) = -f(z) + 2k\pi, \quad k \in \mathbb{Z},$ with the condition that	$p(-z + 2k\pi) = -p(z) + 2k\pi.$ Let $\phi(z)=-z+2k\pi.$ Then $\phi\circ f=-f+2k\pi=g.$
Also,
	\begin{align*}
		f\circ\phi
		&=f(-z+2k\pi)\\
		&=\sin(-z+2k\pi)+p(-z+2k\pi)\\
		&=-\sin z-p(z)+2k\pi\\
		&=-(\sin z+p(z))+2k\pi\\
		&=-f+2k\pi\\
		&=g.
	\end{align*}
	Hence $\phi\circ f=f\circ\phi=g.$
	Observe that $\phi^2=\mathrm{id}.$ Now
	\begin{align*}
		g^2
		&=(\phi\circ f)\circ(\phi\circ f)\\
		&=\phi\circ f\circ\phi\circ f\\
		&=f^2.
	\end{align*}
	More generally,
	$
	g^{2n}=f^{2n},
	$
	and
	$
	g^{2n+1}
	=\phi\circ f^{2n+1}.
	$ We now show that $I(f)=I(g).$
	It is a simple observation that $|z|\to\infty
	\Longleftrightarrow
	|\phi(z)|\to\infty,$
	Now, if $|f^n(z)|\to\infty$, then both $|g^{2n}(z)|$ and $|g^{2n+1}(z)|$ tend to $\infty$. Hence $	I(f)\subset I(g).$ The reverse inclusion is proved similarly, so $I(f)=I(g).$ Analogously, $K(f)=K(g).$
	Also,
	\begin{align*}
		BU(f)
		&=\mathbb C\setminus (I(f)\cup K(f))\\
		&=\mathbb C\setminus (I(g)\cup K(g))\\
		&=BU(g).
	\end{align*}
Thus, we obtain the equality of the three dynamical sets for the functions \( f \) and \( g \) in this context. This leads to a question raised by T. W. Ng. 

If \( f \) and \( g \) are transcendental entire functions satisfying either \( f \circ g = g \circ f \) or \( f \circ f = g \circ g \), can we always find two distinct points \( z_1 \) and \( z_2 \) such that \( f(z_1) = f(z_2) \) and \( g(z_1) = g(z_2) \)? In other words, do \( f \) and \( g \) identify a common pair of distinct points? For the commuting functions we have considered, the answer to this question is affirmative.

\begin{example}
	Consider \( f(z) = z + e^z \) and \( g(z) = f(z) + 2\pi i \). Here, \( f \) satisfies the property \( f(z + 2\pi i) = f(z) + 2\pi i \). It can be easily seen that \( f \circ g = g \circ f \). Moreover, if \( f(z_1) = f(z_2) \), it automatically follows that \( g(z_1) = g(z_2) \) since \( g = f + 2\pi i \). Thus, in this case, we obtain common pairs. Let us explicitly find the points \( z_1 \neq z_2 \).
	We have
	\begin{align*}
	f(z_1) = f(z_2)
	& \implies z_1 + e^{z_1} = z_2 + e^{z_2}\\
	& \implies z_2 - z_1 = e^{z_1} - e^{z_2} = e^{z_1}\!\left(1 - e^{z_2 - z_1}\right) \tag{1}
	\end{align*}
	Let $w = z_2 - z_1 \neq 0$. Then $(1)$ implies 
	\begin{align*}
		w = e^{z_1}(1 - e^w)
		& \implies e^{z_1} = \frac{w}{1 - e^w}, \quad w \notin 2\pi i\,\mathbb{Z}\\
		& \implies z_1 = \log\!\left(\frac{w}{1 - e^w}\right) + 2n\pi i, \quad n \in \mathbb{Z}, \quad z_2 = z_1 + w.
	\end{align*}
%
	\end{example}

	%
The following lemma provides an expression for the iterates of \( g \) in terms of the iterates of \( f \) and the affine map \( \phi(z) = az + b \).

\begin{lemma} \label{rem1}
	Suppose \( f \) and \( g \) are commuting transcendental entire functions, where \( g(z) = af^p(z) + b \), with \( a, b \in \mathbb{C} \), \( p \in \mathbb{N} \), and \( a \neq 0,1 \). Consider the affine map \( \phi(z) = az + b \). Then, for each \( n \in \mathbb{N} \), the iterates of \( g \) can be expressed as $g^n(z) = \phi^n \circ f^{np}(z).$
\end{lemma}
	\begin{proof}
	Observe that 
	\begin{align*}
	g^2(z)
	&=g(g(z))\\
	&= af^p(af^p(z)+b)+b\\
	&= a^2f^{2p}(z)+ab+b.
\end{align*}
	By induction, we can derive that
	\[
	g^n(z) = a^n f^{np}(z) + b \sum_{k=0}^{n-1} a^k.
	\]
Using the affine map \(\phi(z) = az + b\), the function \(g(z)\) can be expressed as \(g(z) = \phi \circ f^p(z)\). Consequently, since \(\phi^n(z) = a^n z + b \sum_{k=0}^{n-1} a^k\), we have \(\phi^n(f^{np}(z)) = a^n f^{np}(z) + b \sum_{k=0}^{n-1} a^k\), which is the same as \(g^n(z)\), and thus the result follows.
	\end{proof}
	For two permutable entire functions, Poon and Yang established, under certain conditions, their Fatou sets are equal, \cite{poon1998dynamical}.
	 \begin{lemma}\label{lem7} \cite[Lemma 2.1]{poon1998dynamical} Let \(f\) and \(g\) be permutable transcendental entire functions. If \(g(z) = a f(z) + b\), where \(a\) and \(b\) are constants with \(a \neq 0\), then \(F(f) = F(g)\). Furthermore, it follows that \(|a| = 1\). \end{lemma}
	  We now generalize this result to show that the filled Julia set and the escaping set of two permutable entire functions are equal. 
\begin{theorem}\label{th1}
	Suppose \( f \) and \( g \) are commuting transcendental entire functions, where \( g(z) = a f^p(z) + b \) for some \( a, b \in \mathbb{C} \), \( p \in \mathbb{N} \), and \( a \neq 0, 1 \). Then, \( K(f) = K(g) \) and \( I(f) = I(g) \).
\end{theorem}

\begin{proof}
	Suppose \( z \in K(f) \). Then, there exists some constant \( A > 0 \) such that \( |f^n(z)| \leq A \) for all \( n \in \mathbb{N} \). Using Lemma \ref{rem1}, we have \( g^n(z) = \psi^n \circ f^{np}(z) \) for all \( n \in \mathbb{N} \). 
	
	Using observation \ref{lem 7}, \( \psi^n \) maps bounded sets to bounded sets. Consequently, there exists a constant \( M > 0 \) such that \( |g^n(z)| \leq M \) for all \( n \in \mathbb{N} \), which implies that \( z \in K(g) \). Thus, we conclude that \( K(f) \subset K(g) \).\\
		%
	On the other hand, if \( z \in K(g) \), then \( f^{np}(z) = \psi^{-n}(g^n(z)) \) is bounded for all \( n \in \mathbb{N} \). Since \( f \) is continuous, it follows that \( f^k(f^{np}(z)) \) is bounded for all \( n \in \mathbb{N} \) and \( k = 1, 2, \ldots, p-1 \). Consequently, the set \( f^n(z) = \{ f^{np}(z) \} \cup \{ f(f^{np}(z)) \} \cup \dots \cup \{ f^{p-1}(f^{np}(z)) \} \) is also bounded for all \( n \in \mathbb{N} \). From this, we deduce that \( K(g) \subset K(f) \). Hence, the result.
	
	Now, suppose \( z \in I(f) \). This means that \( |f^n(z)| \to \infty \) as \( n \to \infty \). Using observation \eqref{lem 7}, we note that \( \psi \) sends unbounded sets to unbounded sets (and therefore \( \psi^n \) does as well). Consequently, it follows that \( |\psi^n(f^{np}(z))| \to \infty \) as \( n \to \infty \). Therefore, we conclude that \( z \in I(g) \), demonstrating that \( I(f) \subset I(g) \). Similarly, by following the same reasoning, we obtain \( I(g) \subset I(f) \). Thus, we arrive at the conclusion that \( I(f) = I(g) \).
	\end{proof}
	
	%
Using the complementary argument, we now demonstrate that these findings also apply to the bungee set, as described in the following remark.

\begin{remark}
	Suppose \( f \) and \( g \) are commuting transcendental entire functions, where \( g(z) = a f^p(z) + b \) for some \( a, b \in \mathbb{C} \), \( p \in \mathbb{N} \) with \( a \neq 0,1 \). Then the bungee sets of \( f \) and \( g \) are identical. Specifically, if \( z \in BU(f) \), then \( z \in \mathbb{C} \setminus (K(f) \cup I(f)) \). Consequently, \( z \in \mathbb{C} \setminus (K(g) \cup I(g)) \), which implies \( z \in BU(g) \).
\end{remark}
As a consequence of Theorem \refeq{th1}, we find that the boundaries of the sets \( I(f) \) and \( I(g) \) are equal, i.e., \( \partial I(f) = \partial I(g) \), where \( \partial U \) denotes the boundary of the set \( U \). Eremenko, in \cite{eremenko1989iteration}, established that \( \partial I(f) = J(f) \). Consequently, this leads us to the conclusion that \( J(f) = J(g) \). This provides another proof of the equality of the Julia sets of two permutable entire functions under the specified conditions, as discussed in \cite{poon1998dynamical}. Furthermore, our results are more general, as setting \( p = 1 \) yields the result found in \cite[Lemma 2.1]{poon1998dynamical}.
	We now present a family of permutable entire functions for which the Julia sets, the escaping sets, the filled Julia sets, and the bungee sets of \( f \) and \( g \) all coincide.
	
	\begin{proposition}
		Suppose \( f \) satisfies the functional equation \( f(d - z) = d - f(z) \) for some constant \( d \). If we define \( g(z) = d - f(z) \), then \( f \) and \( g \) commute.
	\end{proposition}
	
	\begin{proof}
		The proof follows from Lemma \eqref{lem3}, where we take \( \phi = -z + d \).
	\end{proof}
	
	Such a function \( f \) can be constructed by starting with any odd function \( G \) (i.e., \( G(-w) = -G(w) \)). We define the function \( f \) as follows
	
	\[
	f(z) = \frac{d}{2} + G\left(z - \frac{d}{2}\right).
	\]
	
	By substituting \( d - z \) in place of \( z \), we obtain
		\begin{align*}
			f(d-z)
			&=\frac{d}{2}+G\!\left(d-z-\frac{d}{2}\right)\\
			&=\frac{d}{2}+G\!\left(\frac{d}{2}-z\right)\\
			&=\frac{d}{2}-G\!\left(z-\frac{d}{2}\right)\\
			&=\frac{d}{2}-\left[\frac{d}{2}+G\!\left(z-\frac{d}{2}\right)\right]+\frac{d}{2}\\
			&=d-f(z).
		\end{align*}
	Consequently, the function \( g(z) = d - f(z) \) commutes with \( f \). Additionally, \( f \) satisfies the property \( f(az + b) = af(z) + b \). We will illustrate this with some examples.
	
	\begin{example}\label{ex1}
		Suppose \( f(z) = 1 + \sin(z - 1) \) and \( g(z) = 2 - f(z) = 1 - \sin(z - 1) \) are commuting functions. Let's introduce the affine map \( h(z) = 2 - z \), so that
		\[
		g(z) = 2 - f(z) = h(f(z)).
		\]
		Moreover,
			\begin{align*}
				f(h(z))
				&=1 + \sin(1-z)\\
				&=1-\sin(z-1)\\
				&=g(z).
			\end{align*}
			hence $f\circ h = h \circ f$. Using the commutativity of $f$ and $h$, we have
			\[
			g^n = (h\circ f)^n = h^n \circ f^n.
			\]
			Since $h^2 = \mathrm{id}$, we obtain
			\[
			g^n=
			\begin{cases}
				f^n, &  n \text{ is even},\\[4pt]
				h\circ f^n, &  n \text{ is odd}.
			\end{cases}
			\]
		Since \( h \) is conformal, it preserves the normality of families of functions. Therefore, the family \( (g^n) \) is normal if and only if the family \( (f^n) \) is also normal, which leads to the conclusion that \( F(f) = F(g) \).\\ 
		Exploiting the conformality of \( h \) once more, we can deduce that a point escapes under the iteration of \( g \) if and only if it escapes under the iteration of \( f \). Thus, we have \( I(f) = I(g) \). Similarly, a point belongs to the filled Julia set of \( g \) if and only if it belongs to the filled Julia set of \( f \), which implies \( K(f) = K(g) \).
		In the same manner, we also find that \( BU(f) = BU(g) \).
		\end{example}
		
	More generally, we can define \( f(z) = 1 + \sin(z - 1) \) and \( g(z) = -f^p(z) + 2 \) in such a way that \( f \) and \( g \) commute. By applying the same reasoning as in the previous example, we conclude that the escaping set, the filled Julia set, and the bungee set of \( f \) and \( g \) all coincide, as briefly explained in the remark below.
	
	\begin{remark}
		Let \( I(f) = \{ z \in \mathbb{C} : |f^n(z)| \to \infty \text{ as } n \to \infty \} \). If \( z \in I(f) \), then by definition, \( |f^n(z)| \to \infty \) as \( n \to \infty \). We also have \( g^n(z) = h^n \circ f^{np}(z) \), and since \( h \) is conformal, it follows that \( |g^n(z)| \to \infty \). Therefore, \( I(f) \subseteq I(g) \). By using the same argument, we obtained that \( I(g) \subseteq I(f) \). Thus, we conclude that \( I(f) = I(g) \).
	\end{remark}
		Analogously, we can demonstrate that the filled Julia set and the bungee set of two permutable entire functions are the same.
		
		\begin{example}
			Consider the functions \( f(z) = 1 + (z - 1)e^{(z - 1)^2} \) and \( g(z) = -f(z) + 2 \). It is clear that both \( f \) and \( g \) are members of \( \mathcal{S} \) (recall that the $Speiser \hspace{.2cm} class \hspace{.2cm} \mathcal{S}$ consists of those entire functions $f$ which has finite singular values). Now, let’s examine the affine map \( h(z) = 2 - z \). A straightforward computation reveals that \( f(h(z)) = h(f(z)) = g(z) \). Moreover, we can see that \( f \circ g = g \circ f \). As illustrated in example \refeq{ex1}, it can be observed that \( F(f) = F(g) \). By applying the same reasoning as in the previous example, we conclude that the escaping set, the filled Julia set, and the bungee set of both \( f \) and \( g \) are all identical.
		\end{example}
		
		The next lemma from \cite{yang} will be useful in proving the upcoming theorem.
		%
		\begin{lemma} \label{lem 1}
			\cite[Lemma 1]{yang} Let \( f \) and \( g \) be two distinct transcendental entire functions that commute, and let \( Q \) be a non-constant polynomial. Suppose that
			\[ Q(g) = aQ(f) + b \]
			for some \( a \neq 0 \) and \( b \in \mathbb{C} \). Then, for every integer \( n \geq 1 \) and every \( z \in \mathbb{C} \), we have
			\[
			Q\bigl(g^n(z)\bigr) = a^n Q\bigl(f^n(z)\bigr) + b(a^{n-1} + a^{n-2} + \dots + a + 1).
			\]
		\end{lemma}
		
		\begin{proof}
			We provide an alternative proof of this lemma using a recurrence relation argument. Define \( A_n(z) = Q(g^n(z)) \). We begin with the relationship
			\begin{equation}\label{eq1}
				Q(g(z)) = aQ(f(z)) + b.
			\end{equation}
			Now post-multiply both sides of equation \eqref{eq1}, by \( g^{n-1} \).
			\[
			Q(g \circ g^{n-1}(z)) = aQ(f \circ g^{n-1}(z)) + b.
			\]
			Using commutativity of $f$ and $g$, we have $Q(g^n(z)) = aQ(g^{n-1} \circ f(z)) + b.$
			Thus, we can express it as $A_n(z) = aA_{n-1}(f(z)) + b.$\\
			For \( n=1 \), we find
			\[
			A_1(z) = aA_0(f(z)) + b = aQ(f(z)) + b.
			\]
			For \( n=2 \),
			\[
			A_2(z) = aA_1(f(z)) + b = a(aQ(f^2(z)) + b) + b = a^2Q(f^2(z)) + ab + b.
			\]
		Using mathematical induction, we can derive the following
		
		\[
		Q(g^n(z)) = a A_{n-1}(f(z)) + b = a^n Q(f^n(z)) + b \sum_{k=0}^{n-1} a^k.
		\]
		
		If \( a \neq 1 \), we can further simplify this expression as follows
		
		\begin{align}
			Q(g^n(z)) = a^n Q(f^n(z)) + b \frac{a^n - 1}{a - 1}.
		\end{align}
		
		This concludes the proof of the lemma.
			\end{proof}
		To prove the subsequent theorem, we will utilize the lemma above.
	
		\begin{theorem} \label{Th1}
			Let \( f \) and \( g \) be two distinct permutable transcendental entire functions, and let \( Q \) be a non-constant polynomial. Suppose \( Q(g) = aQ(f) + b \) with \( a \neq 0 \) and \( b \in \mathbb{C} \). If \( |a| = 1 \) and \( a \neq 1 \), then the following statements hold
			\begin{enumerate}
				\item \( I(f) = I(g) \),
				\item \( K(f) = K(g) \),
				\item \( BU(f) = BU(g) \).
			\end{enumerate}
		\end{theorem}
		
		\begin{proof}
			We begin by showing that \( I(f) \subseteq I(g) \). Take any \( z \in I(f) \). This means that \( |f^n(z)| \to \infty \) as \( n \to \infty \). Since \( Q \) is a non-constant polynomial, it maps unbounded sets to unbounded sets. Therefore, it follows that \( |Q(f^n(z))| \to \infty \) as \( n \to \infty \). Now, by using Lemma \eqref{lem 1}, we have
			\begin{align*}
				|Q(g^n(z))| 
				&= \left|a^nQ(f^n(z)) +b\frac{a^n-1}{a-1}\right| \\
				&\geq |a^n||Q(f^n(z))|- \left|b\frac{a^n-1}{a-1}\right|\\
				&\geq |Q(f^n(z))|- 2\frac{|b|}{|a-1|},
			\end{align*}
		We have used the fact that \( |a|=1 \) and the inequality \( \left|\frac{a^n-1}{a-1}\right| \leq \frac{2}{|a-1|} \). This shows that \( |Q(g^n(z))| \to \infty \) as \( n \to \infty \), which in turn implies that \( |g^n(z)| \to \infty \) as \( n \to \infty \). Therefore, \( z \in I(g) \), leading to the conclusion that \( I(f) \subseteq I(g) \). 
		By following a similar argument, we can also show that \( I(g) \subseteq I(f) \). As a result, we conclude that \( I(f) = I(g) \).
		The equality of \( K(f) \) and \( K(g) \) can be proved using a similar approach as that of \( I(f) \) and \( I(g) \).
		We will now demonstrate the equality of \( BU(f) \) and \( BU(g) \), drawing an analogy from the equalities of \( I(f) \) with \( I(g) \) and \( K(f) \) with \( K(g) \). To do this, consider an element \( z \in BU(f) \). This implies that there exist at least two subsequences, denoted as \( \{m_k\} \) and \( \{n_k\} \), such that \( |f^{m_k}(z)| \leq R \) for all \( k \in \mathbb{N} \), while \( |f^{n_k}(z)| \to \infty \) as \( k \to \infty \).
		
		Additionally, there exists a constant \( R' > 0 \) such that \( |Q(f^{m_k}(z))| \leq R' \) for all \( k \in \mathbb{N} \), and \( |Q(f^{n_k}(z))| \to \infty \) as \( k \to \infty \). Now, 
			\begin{align*}|Q(g^{m_k}(z))|
				& = \left|a^{m_k}Q(f^{m_k}(z)) +b\frac{a^{m_k}-1}{a-1}\right| \\
				& \leq |a^{m_k}||Q(f^{m_k}(z))| + \left|b\frac{a^{m_k}-1}{a-1}\right| \\
				& \leq |Q(f^{m_k}(z))| + 2\frac{|b|}{|a-1|}.
			\end{align*} 
	
			As the sequence \(\{Q(f^{m_k})\}\) is bounded, we conclude that \(\{Q(g^{m_k}(z))\}\) is also bounded, which implies that \(\{g^{m_k}(z)\}\) is bounded as well. Also,
			\begin{align*}|Q(g^{n_k}(z))| 
			&= \left|a^{n_k}Q(f^{n_k}(z)) +b\frac{a^{n_k}-1}{a-1}\right| \\
			& \geq |a^{n_k}||Q(f^{n_k}(z))| - \left|b\frac{a^{n_k}-1}{a-1}\right| \\
			& \geq |Q(f^{n_k}(z))| - 2\frac{|b|}{|a-1|}. \end{align*}
			As \(k\) approaches infinity, \(|Q(f^{n_k}(z))|\) tends to infinity. Therefore, it follows that \(|Q(g^{n_k}(z))|\) also tends to infinity. Consequently, we find that \(|g^{n_k}(z)|\) tends to infinity as well. This leads us to conclude that \(BU(f) \subseteq BU(g)\).
			
			By applying similar reasoning, we can also establish that \(BU(g) \subseteq BU(f)\). Hence, the result.
			
		\end{proof}

		%
		
		%

	\end{document}